\documentclass[11pt]{article}
\usepackage{amsmath,amssymb,amsbsy,amsfonts,amsthm,latexsym,
               amsopn,amstext,amsxtra,euscript,amscd,color}
\usepackage{mathtools}
\usepackage[draft]{todonotes}
\def \F {{\mathbb F}}

\def \V {{\mathbb V}}

\newtheorem{theorem}{Theorem}
\newtheorem{thm}[theorem]{Theorem}

\newtheorem{prop}[theorem]{Proposition}
\newtheorem{remark}[theorem]{Remark}

\newtheorem{corollary}[theorem]{Corollary}
\newtheorem{conj}[theorem]{Conjecture}

\def\xx{{\bf x}}

\def\00{{\bf 0}}
\def\11{{\bf 1}}
\def\+{\oplus}

\def \F {{\mathbb F}}

\def \V {{\mathbb V}}

\def\wt{{\rm wt}}

\begin{document}

\title{\huge\bf
\textrm{Bisecting binomial coefficients} }

\author{\Large Eugen J. Iona\c scu$^1$, Thor Martinsen$^2$, Pantelimon St\u anic\u a$^2$\\
\vspace{0.1cm} \\
\small $^1$Department of Mathematics, \\
\small Columbus State University\\
\small Columbus, GA 31907\\
\small Email: {\tt \{math\}@ejionascu.ro}\\
\small    \\
\small $^2$Department of Applied Mathematics, \\
\small Naval Postgraduate School\\
\small  Monterey, CA 93943-5212, U.S.A.\\
\small Email: {\tt \{tmartins,pstanica\}@nps.edu}
}

\date{\today}
\maketitle
\thispagestyle{empty}

\begin{abstract}
In this paper, we deal with the problem of bisecting binomial coefficients. We find many (previously unknown) infinite classes of integers which admit nontrivial bisections, and a class with only trivial bisections. As a byproduct of this last construction, we show conjectures $Q2$ and $Q4$ of Cusick and Li~\textup{\cite{CL05}}. We next find several bounds for the number of nontrivial bisections and further compute (using a supercomputer) the exact number of such bisections  for~$n\leq 51$.
\end{abstract}

\section{Introduction}

  In the pursuit of constructing symmetric Boolean functions with various cryptographic properties (resilience, avalanche features), Mitchell~\cite{Mi90}, Gopalakrishnan et al.~\cite{GHS93}, von zur Gathen and Roche~\cite{GR97}, as well as Cusick and Li~\cite{CL05}, among others, study a seemingly ``innocent'' problem, namely the binomial coefficients bisection (BCB), which we shall describe below.

 The connection between symmetric Boolean functions and binomial coefficients is rather immediate. Let $\V_n$ be an $n$-dimensional vector space over the two-element field $\F_2$. A Boolean function $f:\V_n\to\F_2$ is symmetric if  its output value $f(\xx)$ only depends upon the (Hamming) weight  of its input, $\wt(\xx)$ (number of nonzero bits of $\xx$). Since there are $\binom{n}{w}$ vectors $\xx$ of weight $\wt(\xx)=w$, then $f$ is constant on each such set of vectors. Thus, $f$ can be ``compressed'' into an $n+1$ vector of values corresponding to each partition class of cardinality $\binom{n}{w}$, $0\leq w\leq n$. Now, if one further imposes balancedness on $f$ (in addition to symmetry), that is its weight is $\wt(f)=2^{n-1}$, then it follows that one  also has to have a two set partition  $I,J$, of these binomial coefficients $\binom{n}{w}$ so that the function $f$ has value $b\in\{0,1\}$ on the vectors of weight in $I$ and value $\bar b$ on vectors in $J$. Thus, we are prompted in studying these splitting (bisections) of binomial coefficients, and that is the subject of this paper.

If $\displaystyle \sum_{i=0}^n \delta_i\binom{n}{i}=0$, $\delta_i\in\{-1,1\}$, then we call $[\delta_0,\ldots,\delta_n]$ a solution of the (BCB) problem. So, the (BCB) problem consists in finding all these solutions
(the set of all solutions will be denoted by ${\cal J}_n$) and in particular the number of all such solutions, which we will be denoting by $J_n$.
 Certainly, for such a solution, letting $I=\{i\,|\, \delta_i=1\}$ and  $J=\{i\,|\, \delta_i=-1\}:=\overline{I}$, we obtain a {\em bisection} $\displaystyle \sum_{i\in I}\binom{n}{i}=\sum_{i\in J}\binom{n}{i}=2^{n-1}$. Conversely, having a bisection we can reconstruct the solution of (BCB), that it came from, in the previous construction. So, in what follows we are going to use either one of the these descriptions of a solution of the (BCB) problem.

By the binomial theorem  $\sum_{i} (-1)^n\binom{n}{i}=(1-1)^n=0$, so $\pm [1,-1,1,-1,\ldots ]$ is always a solution of (BCB), i.e., we have at least two solutions for every~$n$~($J_n\ge 2$).
We also observe (see also~\cite{CL05})  that if  $n$ is odd then $$[\delta_0,\ldots,\delta_{(n-1)/2},-\delta_{(n-1)/2},\ldots,-\delta_0]$$ with $\delta_i\in \{-1,1\}$ arbitrary chosen, give $2^{(n+1)/2}$ solutions (that include the ones we mentioned before, so $J_{2n-1}\ge 2^n$). These are all called {\em trivial} solutions~\cite{CL05}.

There are sporadic situations when nontrivial solutions do appear. For instance, when $n\equiv 2\pmod 6$, because of the identity $$\binom{n}{k}=2\binom{n}{k-1}=\binom{n}{k-1}+\binom{n}{n-k+1},$$
 \noindent where $k=\frac{n+1}{3}$ being odd, nontrivial solutions appear by moving the above terms from the equality  $$\sum_{i\ odd} \binom{n}{i}=\sum_{i\ even} \binom{n}{i},$$
 \noindent from one side to the other. For example, if $n=8$ we have
$$1+\underline{28}+70+\underline{28}+1=8+\underline{56}+56+8\Rightarrow 1+\underline{56}+70+1=8+\underline{28}+\underline{28}+56+8.$$

\noindent This implies that $[1,-1,-1,1,1,-1,-1,-1,1]$ is a solution for (BCB) problem.
\noindent Besides these type of examples, all that is known about the bisection of binomial coefficients, are mostly computational results (see ~\cite{Mi90,GHS93,GR97,CL05}).

\section{A general approach and  an upper bound }

The well-known  formula from trigonometry $$\cos \alpha\cos \beta=\frac{1}{2}[\cos(\alpha+\beta)+\cos(\alpha-\beta)],\ \alpha,\beta\ \in \mathbb R,$$
can be generalized easily (by induction on the number of angles) in the following way.  For $x_1$, $x_2$, ..., $x_m$ arbitrary real numbers, we have
$$\cos x_1\cos x_2\cdots\cos x_m=\frac{1}{2^{m-1}}\sum \cos ( x_1\pm x_2\pm \cdots \pm x_m),$$
\noindent where the sum is over all possible choices of signs $+$
and $-$. This shows that the number of solutions (all possible choices of signs) of the equation $x_1\pm x_2\pm \cdots \pm x_m=0$ (where $x_i$'s are positive integers) is given by the
formula

$$\frac{2^{m-1}}{2\pi} \int_{-\pi}^{\pi} \cos (x_1t)\cos (x_2t)\cdots\cos (x_mt) dt,$$

\noindent or, since the integrant is an even function,

$$\frac{2^{m-1}}{\pi} \int_{0}^{\pi} \cos (x_1t)\cos (x_2t)\cdots\cos (x_mt) dt.$$

Changing the variable, $t=\pi s$, we can apply this to the bisection of binomial coefficients, and immediately infer the next formula for $J_n$.

\begin{thm}
\label{thm-bisol}
The number of binomial coefficients bisections for fixed $n$ can be computed with the following formula
\begin{equation}\label{eq1}
J_n=2^{n+1}  \int_0^1 \prod_{j=0}^n \cos\left(\pi   \binom{n}{j} s\right)d\,s.
\end{equation}
\end{thm}

We certainly could have used  the below result of Freiman~\cite{Fr80} (see also~\cite{AI14,B82,CF89,Dr88}; seemingly, Drimbe~\cite{Dr88} was unaware of Freiman's work), but we preferred our elementary approach. We mention it here, though, since we will need it later in the paper.
\begin{thm}
\label{freiman}
Let $A=\{a_1,a_2,\ldots,a_N\}$ and $b\leq \frac{1}{2}\sum_{i=1}^N a_i$. The number of Boolean solutions for the equation
\[
\sum_{i=1}^N a_i x_i=b,\  x_i\in\{0,1\}
\]
is precisely $\displaystyle \int_0^1 e^{-2\pi i x b}\prod_{j=1}^N \left(1+e^{2\pi i x a_j}\right)d\,x$.
\end{thm}

Let us denote by $ES(x_1,x_2,...,x_m)$ the number of all solutions of the equation $\pm x_1\pm x_2\pm \cdots \pm x_m=0$. As we have shown, we have

\begin{equation}\label{eq1general}
ES(x_1,x_2,...,x_m)=\frac{2^{m}}{\pi}  \int_0^{\pi} \prod_{j=1}^m \cos\left(x_j t \right)d\,t .
\end{equation}

 In \cite[p. 441]{Rog}, it is shown that for every $k\in \mathbb N$, we have the formula
 $$\int_0^{\pi/2}(\sin t)^k dt=\int_0^{\pi/2}(\cos t)^k dt
 =\begin{cases} \frac{(k-1)!!}{k!!}\frac{\pi}{2} \ \text{if $k$ is even}\\ \\
 \frac{(k-1)!!}{k!!}\ \ \text{if $k$ is odd},
 \end{cases}$$

\noindent where $k!!=k(k-2)\cdots$ (the product of all  integers  $\leq k$ having the same parity as $k$).

A generalization of the Integral H\"older Inequality can be stated in the following way:
given $f_1$, $f_2$,..., $f_m$ functions in $L^{2m}(X)$ ($X$ is a measure space) we have
$$||\prod_{j=1}^m f_j||_2\le \prod_{j=1}^m ||f_j||_{2m},$$
where $\|\cdot\|_p$ is the usual $p$-norm on the measure space $X$.

\noindent Putting these two ingredients together and using Cauchy-Schwartz Inequality,  we obtain
\allowdisplaybreaks[4]
\begin{align*}
ES(x_1,x_2,...,x_m)&\le \frac{2^{m}}{\pi}  \int_0^{\pi} \prod_{j=1}^m |\cos\left(x_j t \right)|d\,t \\
&\le
\displaystyle \frac{2^{m}}{\sqrt{\pi}}\left[\int_0^{\pi} \prod_{j=1}^m |\cos\left(x_j t \right)|^2d\,t \right]^{1/2}\\
&\le
\frac{2^{m}}{\sqrt{\pi}} \prod_{j=1}^m \left(\int_0^{\pi} |\cos\left(x_j t \right)|^{2m}\right)^{1/{(2m)}}dt.
\end{align*}

\noindent  But for $x_j\in \mathbb N$, we have
 \begin{align*}
\int_0^{\pi}  |\cos\left(x_j t \right)|^{2m}dt& =\frac{1}{x_j}\int_0^{x_j\pi} |\cos s|^{2m}ds=\int_0^{\pi} |\cos s|^{2m}ds\\
&=2\int_0^{\pi/2} |\cos s|^{2m}ds=\frac{(2m-1)!!}{(2m)!!}\pi
\end{align*}

Hence, we obtained the following result.
\begin{thm}\label{upperbounds} Given $x_1$, $x_2$, ..., $x_m$ arbitrary positive integers, the following estimates hold
\begin{equation}
\label{upperboundeq}
ES(x_1,x_2,...,x_m)\le 2^{m}\left(\frac{(2m-1)!!}{2m!! }\right)^{1/2}=\left(2\binom{2m-1}{m} \right)^{1/2}.
\end{equation}
In particular,
\begin{equation*}
%\label{upperboundeq2}
J_m \le \left(2\binom{2m+1}{m+1} \right)^{1/2}.
\end{equation*}
\end{thm}

\begin{remark}
 We know already  that there are more than $2^{(n+1)/2}$ bisections for odd $n$ and at most $2^{n+1}$ possible choices. Our Theorem~\textup{\ref{upperbounds}} implies that the quotient between the solutions set size and the size of  possible solutions space is $\frac{1}{\sqrt{\pi (n+1)/2}}\to 0$, as $n\to\infty$, which certainly was expected.
\end{remark}

 \subsection{A more detailed analysis}

We start with the case of $n$ odd.
As in~\cite{BW15} we will be using the inequality
\[
|\cos(\pi x)|^2\leq \exp(-\pi^2\|x\|^2),
\]
which is valid for all real $x$, where $\|x\|$ is the distance to the nearest integer.
For easy writing, for $n$ fixed, we let $b_j=\binom{n}{j}$ and   $B=\lfloor{\frac{n}{2}}\rfloor=\frac{n-1}{2}$, $n$ odd.
Thus,
\begin{equation}
\begin{split}
\label{eq:Jn2}
&2^{-(n+2)}J_n
=\frac{1}{2}\int_0^1\prod_{j=0}^n \cos\left(\pi x\binom{n}{j}\right)=\int_0^{1/2}\prod_{j=0}^B \cos^2\left(\pi x\binom{n}{j}\right)\\
% &= \left( \int_0^{\frac{1}{2 b_B}} + \int_{\frac{1}{2b_B}}^{\frac{1}{2 b_{B-1}}} %+\int_{1/(2b_{B-1})}^{1/(2 b_{B-2})}
%+\cdots+\int_{\frac{1}{2b_1}}^{\frac{1}{2 b_0}}\right) \prod_{j=0}^B\cos^2\left(\pi x\binom{n}{j}
%\right)d\, x\\
&\leq\int_0^{\frac{1}{2 b_B}}  \prod_{j=0}^B\cos^2\left(\pi x\binom{n}{j}\right)d\,x + \int_{\frac{1}{2b_B}}^{\frac{1}{2 b_{B-1}}} \prod_{j=0}^{B-1}\cos^2\left(\pi x\binom{n}{j}\right)d\,x \\
& +\cdots+\int_{\frac{1}{2b_1}}^{\frac{1}{2 b_0}}\prod_{j=0}^{0}\cos^2\left(\pi x\binom{n}{j}\right)d\,x.
\end{split}
\end{equation}
Observe now that, if $\frac{1}{2b_{k+1}}\leq x\leq \frac{1}{2b_k}$, then $\frac{b_j}{2b_{k+1}}\leq xb_j\leq \frac{b_j}{2b_k}\leq \frac{1}{2}$, if $j\leq k$, therefore, $\|xb_j\|=xb_j$, and so, with $B=\lfloor{\frac{n}{2}}\rfloor=\frac{n-1}{2}$ and $S_{n,s}:=\sum_{j=0}^{B-s}  \binom{n}{j}^2 $, we have
 \begin{equation}
 \begin{split}
\label{eq:prod}
 \prod_{j=0}^k\cos^2\left(\pi x\ b_j\right)d\,x
 &\leq \prod_{j=0}^k \exp\left(-\pi^2\left\|xb_j\|\right\|^2 \right)
 =\prod_{j=0}^k\exp\left(-\pi^2 \left(xb_j\right)^2 \right) \\
 &=\exp\left(-\pi^2 x^2\sum_{j=0}^k  b_j^2 \right)
 =\exp\left(-\pi^2 x^2 S_{n,B-k}\right).
% &\quad=\exp\left(-\pi^2 x^2 \left( \binom{2n}{n}-\binom{n}{k+1}{}_3F_2[1,k+1-n,k+1-n;k+2,k+2;1] \right) \right),\nonumber
\end{split}
\end{equation}

Certainly $S_{n,B-k}= \binom{2n}{n}-\binom{n}{k+1}{}_3F_2[1,k+1-n,k+1-n;k+2,k+2;1]$,
using the incomplete sum of powers of binomials coefficients (see~\cite{GKP94}) in terms of the hypergeometric function, but unfortunately this is simply a rewrite of the expression, and will not be very useful in our analysis.

\def\erf{{\rm erf}}

We now let, as it is customary,  $\erf(z)=\frac{2}{\sqrt{\pi}}\int_0^z e^{-t^2}d\,t
 =\frac{2z e^{-z^2}}{\sqrt{\pi}} {}_1F_1\left(1;\frac{3}{2};z^2\right)$,
where ${}_1F_1$ is Gauss' hypergeometric function. It is also known that for $z\gg 1$ (recall that $k!!$ is the double factorial),
\begin{align*}
\erf(z)=\pi^{-1/2}\, \gamma\left(\frac{1}{2},z^2\right)=
  1-\frac{e^{-z^2}}{\sqrt{\pi}}\sum_{k=0}^\infty \frac{(-1)^k (2k-1)!!}{2^k} z^{-2k-1},
  %&=1-\frac{e^{-z^2}}{z\sqrt{\pi}}(1-\frac{1}{2z^2})+O(z^4),
 \end{align*}
 where $\gamma(a,z)=\int_0^z t^{a-1}e^{-t}d\,t$ is the lower incomplete gamma function.
  In particular, under $z\gg 1$, we have~(see~\cite{Bl66}, or any book on probabilities)
\begin{equation}
 \label{eq:erf}
1-\frac{e^{-z^2}}{z\sqrt{\pi}}\left(1-\frac{1}{2z^2}+\frac{3}{4z^4} \right)\leq \erf(z)\leq  1-\frac{e^{-z^2}}{z\sqrt{\pi}}.
\end{equation}
%Under no constraints on $z\geq 0$, we shall use the double inequality~\cite{Bl66}
% \begin{equation}
% \label{eq:erf}
%\left(1-\frac{1}{2 z^2}\right) \frac{e^{-z^2}}{\sqrt{\pi z}} <erfc(z)< \frac{e^{-z^2}}{\sqrt{\pi z}}.
% \end{equation}

 % Recall the elementary inequalities that hold for $x\geq 0$, namely
%\begin{equation}
%\label{eq:e-x}
%e^x\geq 1+x+\frac{x^2}{2}, e^{-x}\leq 1-x+\frac{x^2}{2}.
%\end{equation}

\def\erf{{\rm erf}}

Using the inequalities $e^x\geq 1+x+\frac{x^2}{2}$, $e^{-x}\leq 1-x+\frac{x^2}{2}$ and integrating we can find a better bound, but again, in the interest of simplicity, we roughly bound the decreasing function inside the integral and obtain (assume $0<s\leq B$)
\begin{equation}
\begin{split}
\label{eq:ineq-subint}
&\int_{\frac{1}{2b_{B-s+1}}}^{\frac{1}{2 b_{B-s}}} \exp\left(-\pi^2 x^2 S_{n,s}\right)d\,x=\frac{\erf\left(\frac{\pi \sqrt{S_{n,s}}}{2 b_{B-s}} \right) -\erf\left(\frac{\pi \sqrt{S_{n,s}}}{2 b_{B-s+1}} \right)}{2\sqrt{\pi}\sqrt{S_{n,s}}} \\
%&\quad\leq  \int_{\frac{1}{2b_{B-s+1}}}^{\frac{1}{2 b_{B-s}}} \left( 1-\pi^2 x^2 S_{n,s}+\frac{1}{2} \pi^4 x^4 S_{n,s}^2\right)d\,x\\
%&\quad\leq \left(x-\frac{1}{3}\pi^2 x^3 S_{n,s}+\frac{1}{10}\pi^4 x^5 S_{n,s}^2\right)\Big\rvert_{\frac{1}{2b_{B-s+1}}}^{\frac{1}{2 b_{B-s}}} \\
%  &\quad = \frac{2s}{(n+2s+1)b_{B-s}}-\frac{\pi^2 S_{n,s} (7 n^3 +3 (7 + 18 s) n^2+
% 3 (7 + 36 s + 28 s^2)n+7+ 54 s + 84 s^2 + 72 s^3 )}{24(n+2s+1)^3 b_{B-s}^3}  \\
% &\qquad+ \pi^4 S_{n,s}^2 \left(\dfrac{\splitdfrac{\splitdfrac{31 n^5+5 n^4 (66 s+31)+10 n^3 \left(124 s^2+132 s+31\right)}{+10 n^2
%   \left(264 s^3+372 s^2+198 s+31\right)}}{\splitdfrac{+5 n \left(496 s^4+1056 s^3+744 s^2+264
%   s+31\right)}{+1056 s^5+2480 s^4+2640 s^3+1240 s^2+330 s+31}}}{80\, (n+2 s+1)^5 b_{B-s}^5}\right)\\
&\leq  \exp\left(- \frac{\pi^2 S_{n,s}}{4b_{B-s+1}^2} \right)\left(\frac{1}{2 b_{B-s}}-\frac{1}{2 b_{B-s+1}} \right)\\
  %&\quad\leq \frac{1}{2\,b_{B-s}}-\frac{7\pi^2 S_{n,s}}{24\,b_{B-s}^3}+\frac{\pi^4S_{n,s}^2}{4b_{B-s}^5}\leq
%  \frac{1}{\,b_{B-s}}\left(\frac{1}{2}-\frac{1}{5}+\frac{\pi^4}{4} \right)<\frac{25}{b_{B-s}},\ \text{ for $n$ sufficiently large},
  \end{split}
  \end{equation}

%
%
%We could have expressed these integrals using the hypergeometric functions; for example, the first integral is equal to
%{\footnotesize
%\begin{align*}
% &\frac{s  (32 \pi ^2 (n+2 s+1)^2  (3 (n+1)^2+4 s^2 )
%   \binom{n}{\frac{1}{2} (n-2 s-1)}^2}{} \\
%   &\quad \frac{\cdot(2^n-\binom{n}{\frac{1}{2} (n-2 s+1)}
%   \, _2F_1 (1,\frac{1}{2} (-n-2 s+1);\frac{1}{2} (n-2
%   s+3);-1 ) )}{}\\
%   &\quad \frac{+3 \pi ^4  (40 (n+1)^2 s^2+5 (n+1)^4+16 s^4 )
%   (2^n-\binom{n}{\frac{1}{2} (n-2 s+1)} }{}\\
%  &\quad\frac{\cdot \, _2F_1 (1,\frac{1}{2} (-n-2
%   s+1);\frac{1}{2} (n-2 s+3);-1)){}^2+384 (n+2 s+1)^4
%   \binom{n}{\frac{1}{2} (n-2 s-1)}^4 )}{192 (n+2 s+1)^5
%   \binom{n}{\frac{1}{2} (n-2 s-1)}^5}
%\end{align*}
%}
%but as you can see, this is not very useful for our goal.

We now need to estimate~\eqref{eq:ineq-subint}.
While it is known~\cite{FL05} (see also, Polya and Szeg\"o~\cite[Vol. 1, Prob. 40, P. 42]{PZ72}) that
\begin{align*}
\sum_{j=0}^n \binom{n}{j}^r \sim \left(2^n \sqrt{\frac{2}{\pi n}}\right)^r \sqrt{\frac{\pi n}{2r}},
\end{align*}
as well as the asymptotic for the incomplete sum of powers of binomials
(we let $I:=\{j\in \mathbb{N}\,|\, -a\sqrt{\frac{n}{4}}+\frac{n}{2}\leq j\leq a\sqrt{\frac{n}{4}}+\frac{n}{2}$)
\begin{align*}
\sum_{j\in I} \binom{n}{j}^r\sim 2^{nr}\, \sqrt{\frac{n}{4}}\, \left(\frac{\pi n}{2}\right)^{-r/2} \int_{-a+\frac{2}{\sqrt{n}}}^{-a+\frac{2}{\sqrt{n}}} e^{-rt^2/2}d\, t,
\end{align*}
again, in the interest of simplicity, letting $\alpha_s:=\frac{n}{\lfloor{\frac{n}{2}}\rfloor-s}$, we prefer to use the estimate
$\displaystyle \sum_{k=0}^{\lfloor{\frac{n}{2}}\rfloor - s} \binom{n}{k} =2^{n\left(H(\alpha_s)+o(1)\right)}$,
where $H(\alpha)=-\alpha\log_2(\alpha)-(1-\alpha)\log_2(1-\alpha)$ is the binary entropy function,
which easily implies the inequality %(if $n\geq 5$ I can replace $2^n$ by $2^{n-1}$)
$\displaystyle \sum_{j=0}^{\lfloor{\frac{n}{2}}\rfloor - s} \binom{n}{j} < 2^n e^{-2s^2/n}$ for $0 \leq s  \leq \lfloor{\frac{n}{2}}\rfloor$,
rendering the bounds for $S_{n,s}=\sum_{j=0}^{\lfloor{\frac{n}{2}}\rfloor-s} \binom{n}{j}^2$,
\begin{align*}
\frac{1}{B-s+1} \left(\sum_{k=0}^{\frac{n}{2} - s} \binom{n}{k} \right)^2< S_{n,s}< \binom{n}{\lfloor{\frac{n}{2}}\rfloor-s}\sum_{j=0}^{\lfloor{\frac{n}{2}}\rfloor-s} \binom{n}{j}<2^n e^{-\frac{2s^2}{n}} \binom{n}{\lfloor{\frac{n}{2}}\rfloor-s},
\end{align*}
or the  simpler
\begin{align*}
\frac{1}{B+1}2^{2n\left(H(\alpha_s)+o(1)\right)}\leq \frac{1}{B-s+1} 2^{2n\left(H(\alpha_s)+o(1)\right)}< S_{n,s}<2^n e^{-\frac{2s^2}{n}} \binom{n}{\lfloor{\frac{n}{2}}\rfloor-s},
\end{align*}
 the lower bound being obtained by the Cauchy-Schwarz inequality.
 We can certainly remove the dependence on $o(1)$ by using the inequalities
 \[
 \frac{1}{\sqrt{8n\alpha_s(1-\alpha_2)}} 2^{n H(\alpha_s)}\leq \sum_{k=0}^{\frac{n}{2} - s} \binom{n}{k} \leq 2^{n H(\alpha_s)}.
 \]

 Thus,~\eqref{eq:ineq-subint} becomes (using $\binom{n}{k+1}-\binom{n}{k}=\frac{n-2k}{k+1}\binom{n}{k}$)
 \begin{align*}
 &\int_{\frac{1}{2b_{B-s+1}}}^{\frac{1}{2 b_{B-s}}} \exp\left(-\pi^2 x^2 S_{n,s}\right)d\,x
 \leq \exp\left( -\frac{\pi^2 2^{2n (H(\alpha_s)+o(1))}}{4 (B+1) b_{B-s+1}^2}\right) \frac{b_{B-s+1}-b_{B-s}}{2b_{B-s}b_{B-s+1}}\\
 &\leq \exp\left( -\frac{\pi^2 2^{2n (H(\alpha_s)+o(1))}}{4 (B+1) b_{B-s+1}^2}\right) \frac{n-2B+2s-1}{2(B-s+1)b_{B-s+1}}\\
 &= \exp\left( -\frac{\pi^2 2^{2n (H(\alpha_s)+o(1))}}{4 (B+1) b_{B-s+1}^2}\right) \frac{s}{(B-s+1)b_{B-s+1}}.
 \end{align*}

We next estimate   our integral on  $[0,\frac{1}{2b_B}]$, and use the known identity $\sum_{j=0}^{\lfloor{n/2}\rfloor} \binom{n}{j}^2=\binom{2n}{n}$. Thus,
\begin{align*}
&\int_0^{\frac{1}{2 b_B}}  \prod_{j=0}^B\cos^2\left(\pi xb_j\right)d\,x\leq \int_0^{\frac{1}{2 b_B}} \prod_{j=0}^B\exp\left(-\pi^2 \|xb_j\|^2\right)d\,x\\
&= \int_0^{\frac{1}{2 b_B}} \exp\left(-\pi^2x^2  \sum_{j=0}^B b_j^2  \right)
=\int_0^{\frac{1}{2 b_B}} \exp\left(-\pi^2x^2 \binom{2n}{n}  \right)\\
&=\frac{\erf\left(\frac{\pi\sqrt{\binom{2n}{n}}}{2b_B}  \right)}{2\sqrt{\pi \binom{2n}{n}}}
\leq \frac{1}{2\sqrt{\pi \binom{2n}{n}}} \left(1-\frac{2b_B}{\pi\sqrt{\pi \binom{2n}{n}}}\exp\left(-\frac{\pi^2 \binom{2n}{n}}{4b_B^2} \right) \right),
%&\qquad\qquad =\int_0^{\frac{1}{c_n}}  \prod_{j=0}^B\cos^2\left(\pi xb_j\right)d\,x+\int_{c_n}^{\frac{1}{2 b_B}}  \prod_{j=0}^B\cos^2\left(\pi xb_j\right)d\,x=:I_3+I_4.
%\leq\int_0^{\frac{1}{2 b_B}}  \exp\left(-\pi^2 x^2 \binom{2n}{n}\right)d\,x
\end{align*}
using~\eqref{eq:erf}.
%We estimate $I_4$ as before and obtain
%\[
%I_4\leq \left( \frac{1}{2b_B}-c_n\right)\exp\left(- \pi^2 c_n^2 \binom{2n}{n} \right)<\frac{1}{2b_B} \exp\left(- \pi^2 c_n^2 \binom{2n}{n} \right).
%\]
%For $I_3$ we use a different method.  We start with the
% Taylor expansion
%\[
%\log \cos^2 x=- {x^2} -\frac{x^4}{6}+O(x^6).
%\]
%Using \cite[Lemma 3]{BW15}, assuming that $\frac{c_n^2}{\log B}\sum_{j=0}^B \to\infty$, then
%\begin{align*}
%\int_0^{\frac{1}{c_n}}  \prod_{j=0}^B\cos^2\left(\pi xb_j\right)d\,x
%&=
%%\int_0^{c_n} \exp\left( -\pi^2 t^2 \sum_{j=0}^B b_j^2 -\frac{\pi^4t^4}{6}\sum_{j=0}^B b_j^4+O\left(t^6\sum_{j=0}^B b_j^6 \right)  \right)\\
%%&=\frac{1}{\sqrt{2\pi}} \Gamma(1/2)....
%\end{align*}

Next, we consider the case of $n$ being even, but $n$ is not a power of $2$ (this will be treated in the next section).
 Under this assumption, we see that $v_2\left(\binom{n}{ {n}/{2}}\right)\geq 4$.

 As before,  for $n$ fixed, we let $b_j=\binom{n}{j}$ and   $B=\lfloor{\frac{n}{2}}\rfloor=\frac{n}{2}$, $n$ even. Since   $\cos\left(\pi x \binom{n}{n/2}\right)=0$ for $x=\frac{2k+1}{2\binom{n}{n/2}}$, $k\in\mathbb{Z}$, we see that the expression inside the integral of $J_n$, namely,  $\cos\left(\pi x\binom{n}{n/2}\right) \prod_{j=0}^{n/2-1}\cos^2\left(\pi x\binom{n}{j}\right) $
 is positive for $x$ in  $\left[0,\frac{1}{2\binom{n}{n/2}}\right)\cup \left(\frac{3}{2\binom{n}{n/2}},\frac{5}{2\binom{n}{n/2}}\right)\cup\cdots\cup \left(\frac{\binom{n}{n/2}-1}{2\binom{n}{n/2}},\frac{1}{2}\right]$, and negative in $\left(\frac{1}{2\binom{n}{n/2}},\frac{3}{2\binom{n}{n/2}}\right)\cup \left(\frac{5}{2\binom{n}{n/2}},\frac{7}{2\binom{n}{n/2}}\right)\cdots$. Thus, $J_n$ is the area on the first set of intervals minus the area on the second set of intervals.
 Using this observation, we see that the method we used for the case of $n$ odd applies here, as well, and the bound remains the same (with the obvious change for $B$).

Putting all these estimates together, we thus obtain the following result (by abuse, we include the case of Hamming weight 1, since the bound of Theorem~\ref{powerof2} is stronger in that case).
\begin{thm}
\label{newthm-nodd}
Let $\alpha_s=\frac{n}{\lfloor{\frac{n}{2}}\rfloor-s}$, $n\geq 5$. Then
\begin{align*}
&2^{-(n+2)} J_n \leq \frac{\erf\left(\frac{\pi\sqrt{\binom{2n}{n}}}{2\binom{n}{\lfloor{n/2}\rfloor}}  \right)}{2\sqrt{\pi \binom{2n}{n}}}\\
%\frac{1}{2\sqrt{\pi \binom{2n}{n}}} \exp\left(-\frac{\pi^2 \binom{2n}{n}}{4b_B^2} \right)\\
&\quad +
\sum_{s=1}^{\lfloor{n/2}\rfloor-1} \exp\left( -\frac{\pi^2 2^{2n (H(\alpha_s)+o(1))}}{4 (\lfloor{n/2}\rfloor+1) b_{\lfloor{n/2}\rfloor-s+1}^2}\right) \frac{s}{(\lfloor{n/2}\rfloor-s+1)b_{\lfloor{n/2}\rfloor-s+1}}.
\end{align*}
\end{thm}

\begin{remark}
With a little more work, one can find that the expression above is  $O\left(\frac{2^n}{n}\right)$ (in fact,  $J_n\leq \frac{2^{n+2}}{n}$).
\end{remark}

\section{The  $2^n$ case}
\label{sec:pow2}

We now treat the case of binomial coefficients corresponding to a power of~2.
\begin{thm}
\label{powerof2}
If $N=2^{n}$, $n\geq 3$, then $J_{N}\leq 0.3258\cdot  2^{3\cdot 2^{n-2}-2^{\frac{n-3}{2}}}$, as $n\to\infty$.
 \end{thm}
 \begin{proof}
 We first recall Kummer's result (see also, the paper by Granville~\cite{Gr97}), which states that the $p$-adic valuation ($p$ is a prime number) of a binomial coefficient (for any $N,k$) is
\[
v_p\left(\binom{N}{k}\right)=\sum_{i} \frac{k_i+h_i-N_i}{p-1},
\]
where $N_i,k_i,h_i$ are the digits of $N,k,N-k$, respectively, in their base $p$ representations. Equivalently,
$v_p\left(\binom{n}{k}\right)$ is the number of {\em borrows} when subtracting $k$ from $N$ in base $p$
(a result of Kummer rediscovered by Goetgheluck~\cite{Goe87}). When $N=2^n$ and $n,k\geq 1$, this reveals that
\begin{equation}
\label{eq:v2}
v_2\left(\binom{2^n}{k}\right)+v_2(k)=n.
\end{equation}

It may be useful to visualize our method. The 2-adic valuation of the row of the Pascal's triangle corresponding to $N=2^{2m}$ is
the merging of the $k$-th  rows corresponding to the binomial coefficients $\binom{2^n}{2^k(2s+1)}$, $s\geq 0$. Observe that every  row will have twice as many entries as the one above, disregarding $0$-th row corresponding to the endpoints with the 2-adic valuations $0,0$, occurring at
halves of the intervals above, starting with the 2-adic valuation of the middle binomial $\binom{2^{2m}}{2^{2m-1}}$.
For example, if $n=4$, then the tableaux of 2-adic valuations is
{\small
\begin{equation}
\label{eq:visual}
\begin{array}{ccccccccccccccccc}
 0 &  &   &   &   &     &   &   &   &   &   &      &   &   &   &    & 0 \\
   &  &   &   &   &     &   &   & 1 &   &   &      &   &   &   &    &   \\
   &   &   &   & 2 &     &   &   &   &   &      &   & 2 &   &   &      &   \\
   &   & 3 &   &   &     & 3 &   &   &   & 3 &    &   &   & 3 &   &   \\
   & 4 &   & 4 &   & 4 &     & 4 &   & 4 &   & 4 &   & 4    &   & 4 &   \\
\end{array}
\end{equation}
}
which, by merging will become
\[
\left\{v_2\left(\binom{2^4}{k}\right)\right\}_{k=0}^{16}=\{0, 4, 3, 4, 2, 4, 3, 4, 1, 4, 3, 4, 2, 4, 3, 4, 0\}.
\]

While we conjecture that if $n$ is even, the only  possible bisections are (for $n=2m$)
$B_1=\left\{ \binom{2^{n}}{2k} \right\}_{k=0}^{2^{n-1}-1}$ and  $B_2=\left\{ \binom{2^{n}}{2k+1} \right\}_{k=0}^{2^{n-1}}$, we are unable to show that, but we will use an inductive procedure and show that every row (of our visual aid interpretation), except possibly for the last two rows belong to the same ``bin'', say $B_1$, of a bisection.

For easy writing, for $n$ fixed, we let $b_k:=\binom{2^n}{k}$.
Since $v_2(b_{0})=v_2(b_{2^n})=0$, then it is obvious that the endpoint binomial coefficients occur in the same ``bin'',
say, $B_1$, otherwise, the sums of both of these bins is not even, let alone being equal to $2^{n-1}$. We now let $b_0,b_{2^n}\in B_1$.

Next, we argue that for $n\geq 2$, $b_0,b_{2^{n-1}},b_{2^n}$ belong to the same bin, say $B_1$ (observe that $v_2(b_{2^{n-1}})=1$);
otherwise,  $b_0,b_{2^{n-1}}\in B_1$, $b_{2^n}\in B_2$, say. If that is the case,  then
\begin{align*}
2^{N-1}&=\sum_{b_k\in B_1} b_k =b_0+  b_{2^{n-1}}+\sum_{\substack{k\neq 0, 2^{n-1} \\ b_k\in B_1}} b_k ,\\
2^{N-1}&=\sum_{b_k\in B_2} b_k = b_{2^n}+\sum_{\substack{k\neq  2^{n} \\ b_k\in B_2}} b_k ,
\end{align*}
but that is impossible since both  sums are now odd, but $2^{N-1}$ is even.

Assume now that $b_0,b_{2^{n-1}},b_{2^n}\in B_1$. Further, we argue that, if $n\geq 4$, $b_{2^{n-2}},b_{3\cdot 2^{n-2}}$
also belong to $B_1$. We assume below the opposite.

\noindent
{\em Case $1$.}
If $b_{2^{n-2}},b_{3\cdot 2^{n-2}}$ are split between $B_1$, $B_2$, then, without loss of generality
(note that $b_{ 2^{n-2}}=b_{3\cdot 2^{n-2}}$),  we may assume
\begin{align*}
 2^{N-1}&=2b_0+b_{2^{n-2}}+  b_{2^{n-1}}+\sum_{\substack{v_2(k)\neq 0,n, {n-1} \\ b_k\in B_1}} b_k,\\
2^{N-1}&=b_{3\cdot 2^{n-2}}+\sum_{\substack{v_2(k)\neq  {n-2} \\ b_k\in B_2}} b_k,
\end{align*}
which implies that
\[
v_2\left(\frac{1}{4}b_{3\cdot 2^{n-2}}+\frac{1}{4}\sum_{\substack{v_2(k)\neq  {n-2}\\ b_k\in B_2}} b_k \right)
=0\geq N-3\geq 1, \text{  since $n\geq 4$,}
\]
but this is  impossible.

\noindent
{\em Case $2$.}
If $b_{2^{n-2}},b_{3\cdot 2^{n-2}}$ both belong to $B_2$, then
\begin{equation}
\begin{split}
\label{eq:level2}
2^{N-1}&=2b_0+  b_{2^{n-1}} +\sum_{\substack{v_2(k)\neq 0,n, {n-1} \\ b_k\in B_1}} b_k,\\
2^{N-1}&=2b_{2^{n-2}}+\sum_{\substack{v_2(k)\neq  {n-2} \\ b_k\in B_2}} b_k.
\end{split}
\end{equation}
Applying~\cite[Theorem 1]{Gr97}, we see that $b_{2^{n-1}}\equiv 6\pmod {2^4}$,
which implies that $v_2\left(2b_0+ b_{2^{n-1}}\right)=3$, and since $N\geq 16$, then we must have
$ v_2\left(\sum_{\substack{v_2(k)\neq 0, {n-1} \\ b_k\in B_1}} b_k \right)=3$, and so, $b_{t\cdot 2^{n-3}}\in B_1$,
for some odd $t$. Further, there exists also an odd $t'$ such that
$b_{t'\cdot 2^{n-3}}\in B_2$, because otherwise,
$4\leq v_2\left(\sum_{\substack{v_2(k)\neq {n-2} \\ b_k\in B_2}} b_k \right)=v_2(2^{N-1}-2b_{2^{n-2}})=3$, an impossibility.
Thus, $B_1$ must contain $b_{t\cdot 2^{n-3}}$, $t\in I_1^{(3)}\neq \emptyset$, and
$B_2$ must contain $b_{t\cdot 2^{n-3}}$, $t\in I_2^{(3)}=\{1,3,5,7\}\setminus I_1^{(3)}\neq \emptyset$.
 Let $|I_1^{(3)}|=n_1,|I_2^{(3)}|=n_2=4-n_1$. Since $v_2\left(\sum_{t\in I_2^{(3)}} b_{t\cdot 2^{n-3}}+
 \sum_{\substack{v_2(k)\neq {n-2},{n-3} \\ b_k\in B_2}} b_k \right)=v_2(2^{N-1}-2b_{2^{n-2}})=3$,
 we infer that $|I_1^{(3)}|,|I_2^{(3)}|\in\{1,3\}$. Next,
\begin{align*}
  & v_2\left(2^{N-1}-\sum_{\substack{v_2(k)\neq 0,n,{n-1},{n-3} \\ b_k\in B_1}} b_k \right)
 = v_2\left(2+b_{2^{n-1}}+\sum_{t\in I_1^{(3)}} b_{t\cdot 2^{n-3}} \right)\\
 &=
\begin{cases}
4 & \text{ if } I_1^{(3)}=\{3,5,7\}, I_1^{(3)}=\{1,3,5\}, I_1^{(3)}=\{5\},I_1^{(3)}=\{3\},\\
6 & \text{ if }  I_1^{(3)}=\{7\}, I_1^{(3)}=\{1\},\\
7 & \text{ if } I_1^{(3)}=\{1,5,7\}, I_1^{(3)}=\{1,3,7\}.
\end{cases}
\end{align*}
Further, since $I_2^{(3)}=\{1,3,5,7\}\setminus I_1^{(3)}$, then
\begin{align*}
 & v_2\left(2^{N-1}-\sum_{\substack{v_2(k)\neq {n-2},{n-3} \\ b_k\in B_2}} b_k \right)
  = v_2\left(2b_{2^{n-2}}+\sum_{t\in I_2^{(3)}} b_{t\cdot 2^{n-3}} \right)\\
 &=
\begin{cases}
4 & \text{ if } I_2^{(3)}=\{1\}, I_2^{(3)}=\{7\}, I_2^{(3)}=\{1,3,7\},I_2^{(3)}=\{1,5,7\},\\
6 & \text{ if }  I_2^{(3)}=\{1,3,5\}, I_2^{(3)}=\{3,5,7\},\\
7 & \text{ if } I_2^{(3)}=\{3\}, I_2^{(3)}=\{5\}.
\end{cases}
\end{align*}

Now, assuming $N\geq 6$, we argue modulo $2^5$ in~\eqref{eq:level2}. Since $-2b_{2^{n-2}}\equiv 8\pmod {2^5}$, then we must have $I_2^{(3)}\in\{ \{1,3,5\}, \{3,5,7\},\{3\},  \{5\}\}$. For the complement $I_1^{(3)}=\bar I_2^{(3)}$, using~\cite{Gr97}, we compute the residues modulo $2^5$ of the sum of binomial coefficients $b_{t\cdot 2^{n-3}}$, $t\in I_1^{(3)}$ and obtain that the residues are always $24\pmod {2^5}$, which does not equal the residue of $(2b_0+b_{2^{n-1}})\equiv 8\pmod {2^5}$, obtaining a contradiction.

This argument will inductively work up to the $(n-1)$-st row of ~\eqref{eq:visual}, where the $2$-adic valuation of the $b_k$ for every odd $k$ attains its maximum $n$.
We next assume that there are some $b_{2k}$, $k$ odd, that belong to $B_2$, and so,
\begin{align*}
2^{N-1}&=2b_0+b_{2^{n-1}}+(b_{2^{n-2}}+b_{3\cdot 2^{n-2}})+\cdots +\sum_{\substack{v_2(k)=1 \\ b_k\in B_1}} b_k +\sum_{\substack{v_2(k)=0 \\ b_k\in B_1}} b_k \\
2^{N-1}&= \sum_{\substack{v_2(k)=1 \\ b_k\in B_2}} b_k +\sum_{\substack{v_2(k)=0 \\ b_k\in B_2}} b_k .
\end{align*}
Label %$R_k:=\sum_{t=0}^{2^{k-1}-1} b_{(2t+1) 2^{n-k}}$,
$A:=\sum_{\substack{v_2(k)=1 \\ b_k\in B_1}} b_k$, $B:=\sum_{\substack{v_2(k)=0 \\ b_k\in B_1}} b_k$.
It is known~\cite[Theorem 3]{Len03} that the   sum of all binomial coefficients on the $k$-th row of~\eqref{eq:visual} has the 2-adic valuation equal to $2^k-1$, that is, for
\begin{equation}
\label{eq:Rk}
R_k:=\sum_{t=0}^{2^{k-1}-1} \binom{2^n}{(2t+1)2^{n-k}},\quad v_2\left(R_k\right)=2^k-1.
\end{equation}
From~\eqref{eq:Rk}, we know that $v_2(R_k)=2^k-1$.
It is also not difficult to find that $R_n=2^{2^n-1}$ and $R_{n-1}=2^{2^{n-1}-1} \left(2^{2^{n-1}-1}-1 \right)$.
Observe that
\begin{align*}
& \sum_{\substack{v_2(k)=1 \\ b_k\in B_2}}b_k=R_{n-1}-\sum_{\substack{v_2(k)=1 \\ b_k\in B_1}}b_k=R_{n-1}-A, \text{ and }\\
&\sum_{\substack{v_2(k)=0 \\ b_k\in B_2}}b_k=R_{n}-
\sum_{\substack{v_2(k)=1 \\ b_k\in B_1}} b_k=2^{N-1}-B.
\end{align*}
We therefore get
\begin{align}
&2^{2^n-1} = 2b_0+b_{2^{n-1}}+\sum_{k=2}^{n-2} R_k+A+B\nonumber\\
&2^{2^{n-1}-1} \left(2^{2^{n-1}-1}-1 \right)=A+B.\label{eq:eqAB}
\end{align}

While we conjecture that there are only two bisections for $n$ even and $6$ bisections for $n$ odd (supported by the included data), we are unable to show that. Instead, we find an upper bound for $J_{2^{n}}$, which is better than the one given by Theorem~\ref{newthm-nodd}.

We now use Freiman's Theorem~\ref{freiman}, with $M:=2^{n-1}+2^{n-2}=3\cdot 2^{n-2}$  variables (this is the number of binomial coefficients $\binom{2^n}{k}$, where $v_2(k)=0,1$), $b:=R_{n-1}=2^{2^{n-1}-1} \left(2^{2^{n-1}-1}-1 \right)$,  $a_j=\binom{2^n}{2(2j-1)}$, $1\leq j\leq 2^{n-2}$, and $a_j=\binom{2^n}{2(j-2^{n-2})-1)}$, $2^{n-2}+1\leq j\leq 2^{n-2}+2^{n-1}$, to obtain that the number of ways of solutions for the equation~\eqref{eq:eqAB} is (we shall use again H\"older inequality, as well as $a_j=a_{2^n-j}$ below; also, set $b_{j}:=a_j, 1\leq j\leq 2^{n-3}, b_{j}:=a_{j+2^{n-3}}, 2^{n-3}+1\leq j\leq 2^{n-2}+2^{n-3}$)
\begin{align*}
J_b
&= \int_0^1 e^{-2\pi i x b}\prod_{j=1}^M (1+e^{2\pi i x a_j})d\, x
= \int_0^1 e^{-2\pi i x b}\prod_{j=1}^M \left(2 \cos(\pi i x a_j ) e^{\pi i x a_j}\right)d\, x\\
&=2^M \int_0^1 e^{-2\pi i x b}\prod_{j=1}^M \cos(\pi i x a_j ) e^{\sum_{j=1}^M \pi i x a_j}d\, x\\
&=2^M \int_0^1 e^{-2\pi i x b}\prod_{j=1}^M \cos(\pi i x a_j ) e^{ \pi i x (R_{n-1}+R_n)}d\, x\\
&=2^M \int_0^1 e^{-2\pi i x b}\prod_{j=1}^M \cos(\pi i x a_j ) e^{\pi i x (b+2^{2^n-1})}d\, x\\
&=2^M \int_0^1 e^{\pi i x (2^{2^n-1}-b)}\prod_{j=1}^{M/2} \cos^2(\pi i x b_j )d\, x\\
&\leq  2^M \left(\prod_{j=1}^{M/2} \int_0^1 \cos^M(\pi x b_j) \right)^{2/M}=2^M \frac{1}{2^M} \binom{M}{M/2}\\
&\sim \frac{2^M}{\sqrt{\pi M/2}}\sim 0.3258\cdot  2^{3\cdot 2^{n-2}-2^{\frac{n-3}{2}}},
\end{align*}
and the theorem is shown.
\end{proof}
\begin{conj}
We conjecture that
$J_{2^n}=
\begin{cases}
2 &\text{ if $n$ even}\\
6 &\text{ if $n$ odd}\\
\end{cases}.$
\end{conj}

\section{Some computational results and exact counts}

Using the Hamming High Performance Computer (HPC) at the Naval Postgraduate School, and a parallel computer program written in Julia, we were able to verify the computational data of {\cite{CL05,Jef91}} and obtain additional results for the number of bisections $J_n$, for $n\leq 51$ (for $n$ odd we write the number of bisections as $2^{(n+1)/2}+\cdots$ to point out how many are nontrivial), displayed in Table~\ref{table:binbisections}.

A portion of this sequence, for $n\leq 36$, appears as A200147 in the OEIS (Online Encyclopedia of Integer Sequences), as the number $x_n$, $n\ge 1$, of $0$ or $1$ arrays, $[a_0,a_1,...,a_{n}]$, of $n+1$ elements, with zero $n$-difference.
In general, given a sequence $\{a_n\}_{n\ge 1}$ of real or complex numbers, the {\em first difference sequence} $\Delta( a_n)$ is defined as
$\Delta(a_n)=a_{n+1}-a_n$ for all $n\ge 1$. If we just have a list $L=[a_0,a_1,...,a_{n}]$, then the {\em first difference of $L$}, $\Delta(L)$,
is simply the list $\Delta(L)=[a_1-a_0,a_2-a_1,...,a_n-a_{n-1}]$ which has only $n$-items.
The second difference $\Delta^2(a_n)$ is defined as $ \Delta ^{2}(a_{n})=\Delta (a_{n+1})-\Delta (a_{n})$, and we have similar definitions for lists.
To establish the correspondence between the two countings, let us observe that
\begin{equation}
\label{eq2}
\Delta ^{k}(a_{n})=\Delta ^{k-1}(a_{n+1})-\Delta ^{k-1}(a_{n})=\sum _{t=0}^{k}{\binom {k}{t}}(-1)^{t}a_{n+k-t}.
\end{equation}

 %\vspace{0.1in}

\begin{center}
\begin{table} [ht]
\caption{Number of Binomial Coefficients Bisections}{}
%\vskip.2cm
\allowdisplaybreaks[4]
\centering
 \begin{tabular}{||c c | c c | c c||}
 \hline
 $n$ & $J_n$ & $n$ & $J_n$ & $n$ & $J_n$ \\ [0.5ex]
 \hline\hline
 1 & 2 & 18 & 2 & 35 &  $2^{18}+24$ \\
 \hline
 2 & 2 & 19 & $2^{10}$  & 36 & 2 \\
 \hline
 3 & $2^2$ & 20 & 6 & 37 & $2^{19}$ \\
 \hline
 4 & 2 & 21 & $2^{11}$  & 38 & 38 \\
 \hline
 5 & $2^3$ & 22 & 2 & 39 &  $2^{20}$ \\
 \hline
 6 & 2 & 23 & $2^{12}$  & 40 &2 \\
 \hline
 7 & $2^4$ & 24 & 50 & 41 & $2^{21}+15\cdot 2^{11}$  \\
 \hline
 8 & 6 & 25 & $2^{13}$  & 42 & 2 \\
 \hline
 9 & $2^{5}$  & 26 & 6 & 43 & $2^{22}$  \\
 \hline
 10 & 2 & 27 & $2^{14}$  & 44 & 134 \\
 \hline
 11 & $2^{6}$  & 28 & 2 & 45 & $2^{23}$  \\
 \hline
 12 & 2 & 29 & $2^{15}+2^{11}$  & 46 & 2 \\
 \hline
 13 &  $2^{7}+2^4$ & 30 & 2 & 47 &  $2^{24}+2^{20}$ \\
 \hline
 14 & 14 & 31 &  $2^{16}+5\cdot 2^7$ & 48 & 4098 \\
 \hline
 15 &  $2^{8}$ & 32 & 6 & 49 &  $2^{25}$ \\
 \hline
 16 & 2 & 33 &  $2^{17}+2^{14}$ & 50 & 6 \\
 \hline
 17 &  $2^{9}$ & 34 & 130 & 51 &  $2^{26}$\\  [1ex]
 \hline
\end{tabular}
\label{table:binbisections}
\end{table}
%\vskip-.2cm
\end{center}

%\begin{center}
%\begin{table} [ht]
%\caption{Number of Binomial Coefficients Bisections}{}
%\vskip.2cm
%\centering
% \begin{tabular}{||c c | c c | c c||}
% \hline
% $n$ & $J_n$ & $n$ & $J_n$ & $n$ & $J_n$ \\ [0.5ex]
% \hline\hline
% 1 & 2 & 18 & 2 & 35 & 262168 \\
% \hline
% 2 & 2 & 19 & 1024 & 36 & 2 \\
% \hline
% 3 & 4 & 20 & 6 & 37 & 524288 \\
% \hline
% 4 & 2 & 21 & 2048 & 38 & 38 \\
% \hline
% 5 & 8 & 22 & 2 & 39 & 1048576 \\
% \hline
% 6 & 2 & 23 & 4096 & 40 &2 \\
% \hline
% 7 & 16 & 24 & 50 & 41 & 2127872 \\
% \hline
% 8 & 6 & 25 & 8192 & 42 & 2 \\
% \hline
% 9 & 32 & 26 & 6 & 43 & 4194304 \\
% \hline
% 10 & 2 & 27 & 16384 & 44 & 134 \\
% \hline
% 11 & 64 & 28 & 2 & 45 & 8388608 \\
% \hline
% 12 & 2 & 29 & 34816 & 46 & 2 \\
% \hline
% 13 & 144 & 30 & 2 & 47 & 17825792 \\
% \hline
% 14 & 14 & 31 & 66176 & 48 & 4098 \\
% \hline
% 15 & 256 & 32 & 6 & 49 & 33554432 \\
% \hline
% 16 & 2 & 33 & 147456 & 50 & 6 \\
% \hline
% 17 & 512 & 34 & 130 & 51 & \\  [1ex]
% \hline
%\end{tabular}
%\label{table:binbisections}
%\end{table}
%\vskip-.2cm
%\end{center}

 %\vspace{0.1in}

From what we have seen, if $n=8$, $L:=[1,-1,-1,1,1,-1,-1,-1,1]$ is a nontrivial solution for (BCB) problem.
By (\ref{eq2}), the list $$L=[1,1,-1,-1,1,1,-1,1,1]$$ is a solution of $\Delta ^{8}(L)=[0]$ (by alternating signs).
Adding a constant to a sequence does not change its differences $\Delta ^{k}$, and multiplying a sequence by a number, it  is just a multiplicative
factor for all the differences. Hence, the list
\begin{equation}\label{eq3}
\widetilde{L}=(L+1)/2=[1,1,0,0,1,1,0,1,1]
\end{equation}
 is a  $0$ or $1$ array of $9$ elements with
a zero $8$-difference:  $$\Delta(\widetilde{L})= [0, -1, 0, 1, 0, -1, 1, 0], \ \Delta^2(\widetilde{L})= [-1, 1, 1, -1, -1, 2, -1],$$
$\Delta^3(\widetilde{L})= [2, 0, -2, 0, 3, -3]$, $\Delta^4(\widetilde{L})= [-2, -2, 2, 3, -6]$, $\Delta^5(\widetilde{L})=[0, 4, 1, -9]$, $\Delta^6(\widetilde{L})=[4, -3, -10]$,
$\Delta^7(\widetilde{L})=[-7, -7]$, and finally $\Delta^8(\widetilde{L})=[0]$.

The formulas  \eqref{eq2} and \eqref{eq3} give essentially the bijection between the set of solutions of (BCB) problem and the arrays
described in the sequence A200147.
Let us record this observation and fill in the details.
\begin{prop} The number of bisections of the binomial coefficients, $J_n$, is the same as the number of $0$'s or $1$'s arrays, of $n+1$ elements, with zero $n$-difference, i.e.,
$J_n=x_n$.
\end{prop}
\begin{proof} Suppose we have a  solution $L:=[\delta_0,\ldots,\delta_n]$ of the (BCB) problem. Hence,
 $\displaystyle \sum_{i=0}^n \delta_i\binom{n}{i}=0$, $\delta_i\in\{-1,1\}$. From (\ref{eq2}), we see that this is equivalent to
$\Delta ^{n}(\widehat{L})=0$ where $\widehat{L}=[\delta_0,-\delta_1,\ldots,(-1)^n \delta_n]$. As we have observed in the Introduction, adding
a constant to  $\widehat{L}$, does not affect the differences $\Delta^k$, i.e., we still have  $\Delta ^{n}(\widehat{L}+1)=0$. Finally, since
$$\widehat{L}+1=[1+\delta_0,1-\delta_1,\ldots,1+(-1)^n \delta_n]$$ is a list of $2$'s or $0$'s we can divide by $2$ to obtain an array of
$0$'s or $1$'s:   $\widetilde{L}=[(1+\delta_0)/2,(1-\delta_1)/2,\ldots,(1+(-1)^n \delta_n)/2]$ for which we still have $\Delta ^{n}(\widetilde{L})=\Delta ^{n}(\widehat{L})/2=0$.
It is clear that the map $$L\to \widetilde{L}=\frac{1}{2}(\widehat{L}+1)$$ establishes a bijection between the sets in discussion.
\end{proof}

We say that $f$ is $SAC$~\cite{WT85} if complementing any one of the $n$ input bits the output changes with probability exactly one half.   A Boolean function of $n$ variables satisfies the $SAC$ of order $k$ (we say $f$ is $SAC(k)$ -- see~\cite{Fo88}),  $0\leq k\leq n-2$, if whenever $k$ input bits are fixed, the resulting
function of $n-k$ variables satisfies the $SAC$.

In what follows we will show that $J_n=2$ for infinitely many values of $n$, which will imply  conjecture $Q2$ of Cusick and Li~\textup{\cite{CL05}}, hence $Q4$, as well,    and so, there are only four symmetric $SAC(k)$ functions for infinitely many~$n$.

  First, we let $v_2(n)$ be the 2-adic valuation of $n$, that is, the largest power of 2 occurring in the prime
  power factorization of $n$ (we write $2^{v_2(n)}||n$) (we slightly abuse the notation, as it is usually customary to define the $2$-adic valuation as $2^{-v_2(n)}$).

\begin{thm}
\label{primeminusone}
If $p$ is a prime number, then  $J_{p-1}=2$.
 \end{thm}
\begin{proof}
 The statement is obviously true if $p=2$,  so we may assume that $p$ is an odd prime.
We let $n=p-1$ and observe that  $n\equiv -1$ (mod $p$). We want to show that $\binom{n}{j}\equiv (-1)^j$ (mod $p$), for every $j\in \{0,1,\ldots,n\}$.
This is clearly true for $j=0$.
Since, every $j\in \{1,\ldots,n\}$ has an inverse modulo $p$, we have for $j\in \{1,\ldots,n\}$
\begin{align*}
\binom{n}{j} &\equiv\frac{n(n-1)\cdots (n-j+1)}{j!}\\
&\equiv \frac{(-1)(-2)\cdots (-1-j+1)}{j!}\equiv (-1)^j\ \pmod p.
\end{align*}

Hence, if  $ [\delta_0,\ldots,\delta_n]$ a solution of the (BCB) problem
$$0=\sum_{j=0}^n \delta_j \binom{n}{j}\equiv \sum_{j=0}^n (-1)^j \delta_j \ \pmod p.$$
But the number $$\Delta:=\sum_{j=0}^n (-1)^j \delta_j\equiv 0 \ \pmod p$$ is an odd number ($n+1=p$ is an odd prime) satisfying

\begin{equation}\label{eq4}
|\Delta|\le \sum_{j=0}^n  |(-1)^j \delta_j|=\sum_{j=0}^n 1=n+1=p.
\end{equation}

 \noindent  Because $\Delta$ cannot be zero, the only possible values of $\Delta$ are $p$ or $-p$.
Then the equality $|\Delta|=p= n+1$ in \eqref{eq4},  forces $\delta_j=\pm (-1)^j$, for all $j$. Therefore, we have only the two trivial solutions, that is, $J_n=2$.
\end{proof}

  Next, we are going to use the following construction of a transformation on solutions  of the (BCB), denoted here by $\Theta$,  which we are going to call {\it backward map}. Let $n\in \mathbb N$ with $n\ge 2$ and $\delta=[\delta_0,\ldots,\delta_n]$ be a solution of the (BCB) problem. Hence $\displaystyle \sum_{i=0}^n \delta_i\binom{n}{i}=0$ with $\delta_i\in\{-1,1\}$. Using the Pascal binomial identity
$$\binom{n}{k}=\binom{n-1}{k}+\binom{n-1}{k-1},\ \ 1\le k\le n-1, $$
\noindent  we have $\displaystyle \delta_0+\delta_n+\sum_{i=1}^{n-1} \delta_{i}\left(\binom{n-1}{i}+\binom{n-1}{i-1}\right)=0$.  Rearranging terms, we obtain
$$\delta_0+\delta_1+\sum_{i=1}^{n-1} (\delta_{i}+\delta_{i+1})\binom{n-1}{i}=0.$$
If we define $\eta_j=(\delta_j+\delta_{j+1})/2$, $j\in \{0,1,...,n-1\}$, the identity above becomes $\sum_{i=0}^{n-1} \eta_i \binom{n-1}{i}=0$.
Let us denote the map $[\delta_0,\ldots,\delta_n]\to [\eta_0,\eta_1,\ldots,\eta_{n-1}]$ by $\Theta$. If we restrict the domain of this map to solutions for which
$\delta_0=1$ then it becomes a one-to-one map. We observe that $\eta_j\in \{-1,0,1\}$ for all $j$. So, if we have a trivial solution $\delta$ we get $\Theta(\delta)=0$.
Given a sequence $\eta=[\eta_0,\eta_1,\ldots,\eta_{n-1}]=\Theta(\delta)$ for some $\delta$, we see that $\eta_j=1$ forces $\delta_j=1$ and $\delta_{j+1}=1$. Similarly, if
 $\eta_j=-1$, forces $\delta_j=-1$ and $\delta_{j+1}=-1$. Hence we cannot have two consecutive $\eta$'s having a change of signs, i.e., it must go through a zero value.
In fact, the number of zero's between two changes of sign should be odd.  Let us call this property the {\it (IVP) property} since it resembles the Intermediate Value Property in Calculus. It is easy to see that a sequence like that is then in the range of $\Theta$. Having a non trivial solution $\delta' \in {\cal J}_{n-1}$, this leads to two identities $\eta_1$ and $\eta_2$. If one of these vectors has the (IVP) we say that $\delta'$ has the (IVP).
Let us observe that identities like $\binom{n}{k}-\binom{n}{n-k}=0$ have (IVP), if and only if $n$ is even.

\begin{corollary}
\label{primeminustwo}
If $p$ is an odd prime, then  ${\cal J}_{p-2}$ cannot contain nontrivial solutions which have the (IVP)  property.
 \end{corollary}

\begin{proof} If by way of contradiction, we have a nontrivial solution $\delta'$ which has (IVP), then it leads to a nonzero identity $\eta$ which can be lifted up to
$\delta\in {\cal J}_{p-1}$, i.e. $\eta=\Theta(\delta)$. But we have shown that the only solutions in  ${\cal J}_{p-1}$ are the trivial ones. Hence, $\eta=\Theta(\delta)=0$ and so we get into a contradiction.
\end{proof}

This suggests that the only solutions that we can have in ${\cal J}_{p-2}$ are the ones that lead to trivial identities of the form $\binom{n}{k}-\binom{n}{n-k}=0$ or sums of these (which cannot be lifted since
$p-2$ is odd). This explains why, numerically, $J_{p-2}=2^{\frac{p-1}{2}}$ for many primes $p$.

The Julia program we use represents bisection solutions it finds as binary vectors, $\overrightarrow{v_n}$ (see the appendix).  Given the $n^{th}$ row of Pascal's triangle, $\overrightarrow{p_n}$, along with a corresponding bisection, we represent the dot product as: $\overrightarrow{p_n}\cdot \overrightarrow{v_n} = 2^{n - 1}.$ By inspecting the nontrivial solution vectors we observe the fact that the pattern 10011001 occurs in the nontrivial bisection for $n=13$ and so, prompted by that, we search for other cases where we can insert 1001 at position $n-k$ in the first half, as well as in the corresponding position in the second half.\\

Looking at the bisection solution data (see the appendix) we see some other patterns showing up. We will first consider some identities that were pointed out by Jefferies~\cite{Jef91}, and find the complete solutions set for the implied diophantine equations, rendering, yet again other infinite classes of integers admitting nontrivial bisections.
\begin{thm}
\label{thm:nontr}
We have:
\begin{enumerate}
\item If $n=k^2-2$, $k\geq 4$ even, then $J_n\geq 10$, $J_{n-1}\geq 2^{\frac{n+1}{2}}+2^{\frac{n+1}{2}-3}$ $($tight$)$.
\item If $k\equiv 0,1\pmod 3$ and  $n=\frac{F_{4k+1}+2F_{4k}-6}{5}$, then $J_n\geq 2^{\frac{n+1}{2}}+ 2^{\frac{n-3}{2}}$.
\item Let $n=4k^2 + 16k+13, k\geq 0$. Then, there are at least $2^{(n+1)/2-3}$ nontrivial bisections for the binomial coefficients $\left\{\binom{n}{j}\right\}_{0\leq j\leq n}$, and so, $J_n\geq 2^{\frac{n+1}{2}}+2^{\frac{n-1}{2}}$.
    \end{enumerate}
\end{thm}
\begin{proof}
We first consider the identity
\begin{equation}
\label{eq:1st-id}
\binom{n}{x}+\binom{n}{x+2}=2\binom{n}{x+1}.
\end{equation}
By expanding and canceling out the factorials, we obtain the diophantine equation (assume that $n>1$)
\begin{align*}
 n^2 -4nx  + 4 x^2 - 5n  + 8 x+2=0.
\end{align*}
We will take an elementary approach to this equation, and write it as
\begin{align*}
(n-2x)^2-4(n-2x)-n+2= (n-2x-2)^2-n-2=0,
\end{align*}
that is, $n-2x-2=\pm k$ and $n+2=k^2$, $k\in\mathbb{Z}$, and so, we get the integer solutions for~\eqref{eq:1st-id}
\begin{align*}
n=k^2-2, \quad x=\frac{k^2\mp k}{2}-2.
\end{align*}
Note that Jefferies~\cite{Jef91} provides only the even solutions.

Now, we must argue whether these identities will generate nontrivial bisections. As we mentioned previously, the way we use these identities is to transform a trivial bisection into nontrivial ones by interchanging the two sides of the identity, assuming each side occurs in the same bisection. If $n$ odd, recall that a trivial bisection is obtained by taking randomly the first half of the coefficient $\{0,1\}$-vector, and the second half is the complement. However, in our case, these binomials occur in the first half, so this identity will not give us nontrivial bisections. If $n$ is even, we get the two trivial bisections by putting all even indexed binomials in one bin, and all the odd indexed ones in the other bin. Since $x\equiv x+2\not\equiv x+1\equiv n-(x+1)\pmod 2$, then this identity will give us eight more (four such for each choice of the $\mp$ sign) nontrivial bisections (see also~\cite{Jef91}).

We now look at the binomial identity, which while observed in~\cite{Jef91} for $n=13,x=3$, or $x=7$, was not solved there in its full generality:
\begin{equation}
\label{eq:2nd-id}
\binom{n}{x}+\binom{n}{x+3}=\binom{n}{x+1}+\binom{n}{x+2}.
\end{equation}
Equation~\eqref{eq:2nd-id} is equivalent to
\[
  n^2 + 4 x^2 -4nx -7n +12x+6=0.
\]
It turns out that it is as easy as the previous diophantine equation and a similar elementary approach renders the solutions
\begin{align*}
n=k^2-3, \quad x=\frac{k^2\mp k}{2}-3.
\end{align*}
In the case of odd $n$, the situation is different. The idea is to transform a trivial bisection (whose second half $\{0,1\}$-vector is the complement of the arbitrarily chosen first half) by keeping a small vector fixed in the first half (and the  second half), which we show that has equal sum.
For the previous values of $n,x$, we obtain $2\cdot 2^{\frac{n+1}{2}-4}$ many nontrivial bisections (a tight bound as we see from our table, since $J_{13}=2^{\frac{13+1}{2}}+2^4$).

Next, we consider the binomial equation
\begin{equation}
\label{eq:3rd-id}
\binom{n}{x+2}=\binom{n}{x+1}+\binom{n}{x}.
\end{equation}
We point out that the single solution $(n,x)=(117,38)$ provided in~\cite{Jef91} is incorrect, and it should rather be $(n,x)=(103,38)$. In fact, we shall find all solutions to this diophantine equation, although, the method is slightly more complicated than the previous diophantine equations. We do not claim that this equation has not been considered before, but we were not able to find a suitable reference.

From~\eqref{eq:3rd-id} we obtain
\[
 n^2 + x^2 - 3 n x -3n-2=0,
\]
which can be written (multiplying by $20$ so that we have an  equation in integers) as the Pell equation
(for convenience, we take $x\leq n/2$, so $3n>2x$)
\begin{equation}
\label{eq:3rd-id2}
(5n+6)^2-5(3n-2x)^2=-4.
\end{equation}

 Our reason for purposefully disregarding a (well-known to specialists) recurrence identity (namely, $L_n^2-5F_n^2=4 (-1)^n$, where $F_n,L_n$ are the Fibonacci, respectively Lucas numbers, satisfying the same recurrence $F_{m+1}=F_m+F_{m-1}$,  with $F_0=0, F_1=1$, $L_0=2,L_1=1$) is two-fold: it is not obvious that the mentioned identity will render {\em all} solutions to our diophantine equation; secondly, we wish to give yet another proof to that identity via Pell equations theory.

Fortunately, the Pell equation $X^2-5Y^2=-4$ can be solved precisely in a form that is convenient to us~(see~\cite{Ma00, Tek07}). First, observe that $(x_1,y_1)=(1,1)$ is its fundamental solution. Pell equation theory shows that all solutions to~$X^2-5Y^2=-4$ are then $(x_{2n+1},y_{2n+1})$, where
\[
x_{2m+1}+y_{2m+1}\sqrt{5}=\frac{1}{2^{2m}} (1+\sqrt{5})^{2m+1}= 2 \phi^{2m+1},
\]
where $\phi=\frac{1+\sqrt{5}}{2}$ is the golden mean. We now use the identity
\[
\phi^k=\phi F_k+F_{k-1},
\]
therefore,
\[
(x_{2m+1},y_{2m+1})=(F_{2m+1}+2F_{2m},F_{2m+1}),
\]
are all solutions to~$X^2-5Y^2=-4$, and so,
the following solutions for~\eqref{eq:3rd-id2}
\begin{align*}
n&=\frac{F_{2m+1}+2F_{2m}-6}{5},\\
x&=\frac{3F_{2m}-F_{2m+1}-9}{5},
\end{align*}
assuming they are integers. It is rather easy to show that $m$ must be even, say $m=2k$ and so, the general solution to~\eqref{eq:3rd-id} now becomes
\begin{align*}
n&=\frac{F_{4k+1}+2F_{4k}-6}{5},\\
x&=\frac{4F_{4k+1}+3F_{4k}-9}{5}.
\end{align*}
%For example, the first few solutions are $\{14, 4\}, \{103, 38\}, \{713, 271\}, \{4894,\break 1868\}, \{33551, 12814\}, \{229969,   87839\}$, $\{1576238, 602068\}, \{10803703, 4126646\}, \break \{74049689,
%  28284463\}, \{507544126,193864604\}, \{3478759199,  1328767774\},\break \{23843770273, 9107509823\},  \{163427632718,
%  62423800996\}$.
   We are looking for odd values of $n$, which will happen if $F_{4k+1}$ is odd. Using the entry point modulo 2 for the Fibonacci numbers, we infer that $F_{4k+1}$ is odd if $k\equiv 0,1\pmod 3$.

  Certainly, since then for such an odd $n$  we could  ``destroy'' the triviality of a bisection by placing, for $x<n/2$, the binomials $\binom{n}{x},\binom{n}{x+1}$ in one bin and $\binom{n}{x+2}$ in another bin (similarly, for $x>n/2$), we infer that there are more than $2\cdot 2^{\frac{n+1}{2}-3}$ nontrivial bisections in this case.

%There are other ways of expressing these solutions. For example,   with
%\begin{align*}
%\begin{pmatrix}
%u_{m}\\
%v_{m}
%\end{pmatrix}
%=\begin{pmatrix}
%x_1 & D y_1\\
%y_1 & x_1
%\end{pmatrix}^m
%\begin{pmatrix}
%1\\
%0
%\end{pmatrix},
%\end{align*}
% the solutions to  $x^2-5y^2=-4$ are all of the form
%\[
%(x_{2m+1},y_{2m+1})=\left(\frac{u_{2m+1}}{2^{2m}},\frac{v_{2m+1}}{2^{2m}} \right),
%\]
%which will render the following solutions for~\eqref{eq:3rd-id2}
%\begin{align*}
%n&=\frac{1}{5}\left(\frac{u_{2m+1}}{2^{2m}}-6\right),\\
%x&=\frac{3}{5}\frac{u_{2m+1}}{2^{2m+1}}+\frac{v_{2m+1}}{2^{2m+1}}-\frac{9}{5}.
%\end{align*}

Next, we define the operation $\tilde \cdot$ on a $\{0,1\}$-bit block $B$, which outputs the mirror image block $\tilde {B}$. For example, $\widetilde{100}={001}$. Also, recall that $\bar{B}$ is the complement of the block $B$.

 Let $n=2t+1$ and $k$ to be determined later.
 The idea is to start with a trivial bisection $H||\bar{\tilde{H}}$ (where $H$ is a random first block) for $n$  and replace a 4-bit block $k$ bits away from the middle of the sequence) in $H$ and the corresponding 4-bit block in $\bar{\tilde{H}}$ by $1001\stackrel{k}{\overbrace{\cdots}}||\stackrel{k}{\overbrace{\cdots}}1001$ (similarly, by $0110\stackrel{k}{\overbrace{\cdots}}||\stackrel{k}{\overbrace{\cdots}}0110$) to preserve the bisection.

 Here, we force
 %$n=4m^2 + 8m + 1$ and $k:=m-1, t:=2m^2+4m$.
  $n$ to satisfy the following binomial coefficient identity
\begin{align*}
&\binom{n}{t-k-3}+\binom{n}{t-k}+\binom{n}{t+k+1}+\binom{n}{t+k+4}\\
=& \binom{n}{t-k-1}+\binom{n}{t-k-2}+\binom{n}{t+k+2}+\binom{n}{t+k+3},
\end{align*}
which is equivalent to
\begin{align*}
&\binom{n}{t-k-3}+\binom{n}{t-k}= \binom{n}{t-k-1}+\binom{n}{t-k-2}.
\end{align*}
Multiplying the above equation by $\frac{(t-k-3)!(t+k+1)!}{n!}$ we obtain the equation
{\small
\begin{align*}
& \frac{1}{(t+k+2)(t+k+3)(t+k+4)}+\frac{1}{(t-k)(t-k-1)(t-k-2)}\\
=&\frac{1}{(t-k-1)(t-k-2)(t+k+2)}+ \frac{1}{(t-k-2)(t+k+2)(t+k+3)},
\end{align*}
}

\noindent
which  renders the diophantine equation
\begin{align*}
&6 + 8 k + 2 k^2 - t=0,
\end{align*}
therefore, for every value of $k\geq 0$, one can take
$n=2t+1=4k^2+16k+13$, for which there are (at least two) nontrivial bisections. The bound can be improved observing that the first $(n+1)/2-4$ bits can be taken arbitrarily.
\end{proof}

\section{Appendix}

We display below the values of $n\leq 10000$ given by Theorem~\ref{thm:nontr}, for which there are nontrivial bisections, namely,
{\small
\begin{align*}
& 13, 14, 33, 34, 61, 62, 97, 98, 103, 141, 142, 193, 194, 253, 254, 321, 322, \\
&397, 398, 481, 482, 573, 574, 673, 674, 713, 781, 782, 897, 898, 1021, 1022,\\
& 1153, 1154, 1293, 1294, 1441, 1442, 1597, 1598, 1761, 1762, 1933, 1934, \\
& 2113, 2114, 2301, 2302, 2497, 2498, 2701, 2702, 2913, 2914, 3133, 3134, \\
& 3361, 3362, 3597, 3598, 3841, 3842, 4093, 4094, 4353, 4354, 4621, 4622, \\
& 4897, 4898, 5181, 5182, 5473, 5474, 5773, 5774, 6081, 6082, 6397, 6398, \\
&6721, 6722, 7053, 7054, 7393, 7394, 7741, 7742, 8097, 8098, 8461, 8462, \\
& 8833, 8834, 9213, 9214, 9601, 9602, 9997, 9998.
\end{align*}
}

The table which follows contains the complete set of nontrivial bisection solution vectors for $1 \leq n \leq 50$.  In the interest of saving space, we only list the highest lexicographically occurring solutions.  Any additional solutions which a listed solution may yield, can be generated in the following manner:  If a pair of bits are equidistant from the center of the given vector and differ, they may both be complemented to produce a new solution.  Additionally, any solution vector can also be reversed and complemented in its entirety to produce yet another solution.

\label{Appendix}

\begin{center}
 \begin{tabular}{|| c | c | c ||}
 \hline
 $n$ & \# nontrivial sols. & nontrivial sol. vectors \\ [0.5ex]
 \hline\hline
 8 & 4 & 100110001\\
 \hline
 13 & 16 & 11110011001000\\ [1ex]
 \hline
 14 & 4 & 101001101000101\\
   & 8 & 101011100100101\\
 \hline
 20 & 4 & 101010011010100010101\\
 \hline
 24 & 32 & 1000110111011000100010001\\
  & 16 & 1011001111010100101000101 \\
 \hline
 26 & 4 & 101010100110101010001010101 \\
 \hline
 29 & 2048 & 111111110111011000110010000000\\
 \hline
 31 & 512 & 11110110011111100010101000001000\\
    & 128 & 11110110010110011001100000001000\\
 \hline
 32 & 4 & 101010101001101010101000101010101\\
 \hline
 33 & 16384 & 1111111111111001101001000000000000\\
 \hline
 34 & 64 & 10101001110110111010000000110010101\\
    & 32 & 10101001110111101010010000110010101\\
    & 16 & 10101001111100111010000110110010101\\
    & 8  & 10101001111101101010010110110010101\\
    & 8  & 10101010101011011010001010101010101\\
 \hline
 35 & 8  & 101010101010100111001001010101010101\\
    & 16 & 101010101011100111001000110101010101\\
 \hline
 38 & 4  & 101010101010011010101010100010101010101\\
    & 32 & 101111110010111110100011100010011011101\\
 \hline
 41 & 2048  & 111111011110101001111000100100001110100000\\
    & 4096  & 111111011110111001111000100010001110100000\\
    & 8192  & 111111111111001010111001000100100010100000\\
    & 16384 & 111111111111011010111001000010100010100000\\
 \hline
 44 & 4   & 101010101010100110101010101010001010101010101\\
    & 128 & 101011111000111111110110000011011000110110101\\
 \hline
 47 & 1048576 & 111111111111110100111111000001000000100000000000\\
 \hline
 48 & 4096 & 1011001111011011010111010101000000000001000000101\\
 \hline
 50 & 4 & 101010101010101001101010101010101000101010101010101\\
 \hline \hline
\end{tabular}
\label{table: solutions}
\end{center}

\end{document}